\documentclass[11pt]{article}
\usepackage{a4,comment,amsmath,amssymb,amsthm,cite}

 \theoremstyle{plain}
\newtheorem{thm}{Theorem}
\newtheorem{cor}{Corollary}

\theoremstyle{definition}
\newtheorem{rem}{Remark}

\newtheorem{prob}{Problem}
\newtheorem{quest}{Question}

\newtheorem{ack}{Acknowledgements}

\newtheorem{conj}{Conjecture}
\nonstopmode
\date{\vspace{-5ex}} %\nodate

\newcommand{\C}{{\mathbb C}} \newcommand{\Q}{{\mathbb Q}} \newcommand{\Z}{{\mathbb Z}} 
 \newcommand{\N}{{\mathbb N}}
\newcommand{\Span}{\operatorname{span}}
\newcommand{\abs}[1]{{\left| {#1} \right|}}

\renewcommand{\Re}{\operatorname{Re}}
 \newcommand{\Norm}{\operatorname{Norm}}

\author{Johan Andersson\thanks{Email:johan.andersson@his.se \, Address:School of Engineering science, University of Sk\"ovde, Box 408, 541 28 Sk\"ovde, SWEDEN}}
\title{On  questions of Cassels and Drungilas-Dubickas}

\begin{document}
   \maketitle
 
\begin{abstract}
  We answer a question of Drungilas-Dubickas in the affirmative under the assumption of standard conjectures on smooth numbers in polynomial sequences. This gives evidence against the ``Dubickas Conjecture'', which Ka{\v{c}}inskait{\.e} and Laurin{\v{c}}ikas  proved implies universality results for the Hurwitz zeta-function with certain algebraic irrational parameters.

Under these standard conjectures we also  prove some results that confirms observations of Worley relating to a problem of Cassels on the  multiplicative dependence of  algebraic numbers shifted by integers.
\end{abstract}

\section{Universality results for the Hurwitz zeta-function}

The classical universality result for the Hurwitz zeta-function (see e.g. \cite{GarLau}) says that
\begin{thm}Let $0<\alpha<1/2$ or $1/2<\alpha<1$ be either a transcendental number or a rational number  and let $K$ be some compact set with connected complement lying in the strip $K \subset \{s \in \C:1/2<\Re(s)<1 \}$, and suppose that $f$ is any continuous function on $K$ that is analytic  in the interior of $K$. Then 
$$\liminf_{T \to \infty} \frac 1 T \mathop{\rm meas} \left \{t \in [0,T]:\max_{z \in K} \abs{\zeta(z+it,\alpha)-f(z)}<\varepsilon \right \}>0. $$
\end{thm}

If $\alpha=1$ or $\alpha=1/2$ it follows from the Voronin universality for the Riemann zeta-function that the same results are true if $f$ is assumed to be nonvanishing on $K$. One of the main open problems in the theory of Universality of $L$-functions and zeta-function is to prove the same result also whenever $\alpha$ is an algebraic irrational
\begin{conj}
  Theorem 1 is true for all $0<\alpha<1/2$ and $1/2<\alpha<1$
\end{conj}
\noindent The arithmetical results needed to prove Theorem 1 are 
\begin{enumerate}
  \item For $\alpha$ rational: the fact that $\log p$ are linearly independent over  $\Q$. This is used to prove joint universality for the Dirichlet $L$-functions, where      the primes comes from taking logarithms of the $L$-functions. % from which the required result for the Hurwitz zeta-function follows.
  \item For $\alpha$ transcendental. This is somewhat easier since we do not have to take the logarithm and it is sufficient to use the fact that $\log(n+\alpha)$ are linearly independent over  $\Q$.   
\end{enumerate}
When $\alpha$ is  algebraic irrational the problem becomes much harder, since we are likely to have a lot of linear relations between $\log(n+\alpha)$. While we still believe that Conjecture 1 is true, it is our opinion that some substantially new ideas are required for a proof.

We should remark that in some related problem the case of $\alpha$ algebraic irrational has been covered. With respect to the value-distribution of $\zeta(s,\alpha)$ this can be found in \cite{GarLau}. Also  Laurin{\v{c}}ikas and Steuding  \cite{LauSteu} proved that $\zeta(s+it,\alpha)$ admits a limit distribution in the space of analytic functions on $K$. However this is much easier to do and do not imply Theorem 1, since we also have to show that any 
continuous function on $K$, analytic in its interior is contained in this limit distribution.

 \section{Multiplicative dependence of shifted algebraic numbers}
  We say that a set $M \subset \C $ \,  is multiplicative dependent (see Dubickas \cite{Dub}) if there exists distinct $x_1,\ldots,x_n \in  M$  and integers $k_1,\ldots,k_n$ not all zero such that   \begin{gather*}
  %\label{dep}
	\prod_{j=1}^n x_j^{k_j}=1.
 \end{gather*}
If $M$ is not multiplicative dependent we will say that $M$ is multiplicative independent. 
Let  the natural numbers be defined by $\N=\{n \in \Z:n \geq 0\}$, and write the set of shifted natural numbers as $\N+ \alpha =\{n+\alpha:n \in \N \, \}$. It is clear that linear dependence is related to multiplicative dependence.  In particular we see that $\{\log(n+\alpha): n \in \N \, \}\cup\{2\pi i\}$ is linearly dependent over $\Q$ \, if and only if $\N+ \alpha$ is multiplicative dependent. The reason why we adjoin $2\pi i$ to the set is because for complex number field case the logarithm is not well defined. When we adjoin $2 \pi i$ we may choose any branch of the logarithm, for example the principal part. It is easy to see that $\N+\alpha$ is multiplicative independent if $\alpha$ is
 transcendental, since any multiplicative relation 
$$ \prod_{j=1}^m (\alpha+n_j)^{k_j}=1$$
readily gives us a polynomial with integer coefficents, that has $\alpha$ as a zero.

One of the classical results in the field is the result of Cassels \cite{Cassels}, which asserts that we can find a multiplicative independent set of algebraic irrational numbers shifted by natural numbers $M \subset \N+\alpha$ of density $0.51$. A different way to state this is the following:
$$
  \dim \Span_\Q \{n+\alpha: 0 \leq n < x \} \, \geq \, 0.51x+o(x).
$$
 Cassels used this result to show that the Hurwitz zeta function has zeros in any strip  $1<\Re(s)<1+\delta$ for any algebraic irrational parameter $\alpha$. Since this result was previously proved  for transcendental and rational parameters by Davenport-Heilbronn \cite{DavHeil}, this showed that the result was true for any parameter $\alpha$. Worley  \cite{Worley} used a variant of Cassels method to prove that $0.51$ in Cassels result may be replaced by $1-\frac 1 {2d}+O(d^{-3/2})$, where $d$ is the degree of the number field. This improves on Cassels bound for large degree $d$. Drungilas-Dubickas \cite{DubDrun} asked the following question. 

\begin{quest}
 Is the set $\N+\alpha$ multiplicative dependent for every algebraic number $\alpha$?  
\end{quest}

In \cite{Dub}, Dubickas proved that the question has an affirmative answer for the case of quadratic number fields. In \cite{DubDrun} Drungilas-Dubickas  proved the same result for cubic number fields if the algebraic numbers are shifted by rationals instead of natural numbers. These results give some evidence that the answer to the question might be {\em Yes} for arbitrary number fields

While Drungilas-Dubickas did not phrase the question as a conjecture in \cite{DubDrun}, Ka{\v{c}}inskait{\.e} \cite{Kac} and  Ka{\v{c}}inskait{\.e}-Laurin{\v{c}}ikas \cite{KacLau}  stated and called the following conjecture the Dubickas conjecture:
\begin{conj}
 There exists some algebraic number $\alpha$ such that $\N+\alpha$ is multiplicative independent.
\end{conj}
Thus despite the results of Dubickas for quadratic number fields, the Dubickas conjecture is that the answer to Question is {\em No}. For such algebraic numbers $\alpha$,  Ka{\v{c}}inskait{\.e} and Laurin{\v{c}}ikas  used standard methods to prove Theorem 1, Voronin Universality for the Hurwitz zeta-function. At the moment there are no known values of algebraic irrational numbers where Theorem 1 is known to hold.

\section{The Number field sieve}

Mathematicians familiar with the number field sieve might be suspicious of Conjecture 2. These types of relations are exactly what is used in the number field sieve in order to factor integers fast. 
While not proven rigorously the number field sieve is expected to be fast, and there is quite a lot of computational evidence. The number field sieve even requires us to find similar relations for two different polynomials. In particular, in one standard variant of the number field sieve, see e.g. Pomerance \cite{Pomerance}, one of the polynomials chosen is exactly $f(n)=\alpha+n$. 
If Conjecture 2 is true, there would exist some integer and choice of polynomials where the  sieving part of the number field sieve would never finish.
 While our problem is simpler since just one polynomial is required, knowledge of the number field sieve gives us good intuition about the problem at hand.

Similarly to the number field sieve, to find the relations we just consider the algebraic numbers where $\alpha+n$ is without large prime factors, and discard the rest of the numbers. 
This is what is done in the sieving step of the number field sieve. Among these remaining numbers we should be able to, and can under some standard conjectures, find the relations needed to answer our Question.

\section{Smooth numbers}

A smooth number is an integer without large prime factors. Dickman \cite{Dickman} proved that
\begin{gather*}
 \Psi(x,y) \sim \rho(u) x, \qquad (x=y^u),
\end{gather*}
where $\Psi(x,y)$ denote the number of positive integers less than $x$ with all prime factors less than or equal to $y$.  The Dickman-function $\rho(u)$  can be defined for $u \geq 0$ by the initial values $\rho(u)=1$ for $0\leq u \leq 1$ and by the difference-differential equation 
$$
 \rho'(u)=-\frac{\rho(u-1)} u,
$$
for $u>1$. In particular we see that $\rho(u)=1-\log u$ for $1 \leq u \leq 2$. The problem of determining the number of smooth numbers in various sets have many application in analytic and computational number theory, especially at analyzing how fast methods for factorization, such as the Number Field Sieve are expected to work. Good surveys of this area are for example Baker \cite{Baker}, Hildebrand-Tenenbaum \cite{HilTen} and Granville \cite{Granville}. Similarly to prime density estimates we expect that if there is no reason otherwise, we should have the same density of smooth numbers in polynomial sequences  as amongst the natural numbers. This can be quantified by the following conjecture.
\begin{conj} Let $P \in \Z[x]$ be an irreducible polynomial of degree $d$ and let $\Psi_P(x,y)$ be the number of integers in the set  $\{P(n):1 \leq n < x\}$ with all prime factors less or equal to $y$. Then
 $$ \Psi_P(x,y) \sim \rho(du) x, \qquad (x=y^u).$$
\end{conj}
This is a special case of a more general conjecture for several polynomials of Martin \cite{Martin} (see also \cite[p 5]{Granville}). Martin gave numerical evidence supporting this Conjecture and  managed to prove\footnote{It was proved unconditionally that it is of the right order of magnitude in this range by Martin\cite{Martin}. For the several polynomial case, Dartyge-Martin-Tenenbaum \cite{DarMarTen} proved the corresponding result.} it under a quantitative version of Shinzel's hypothesis $H$\footnote{This  is a very strong conjecture on prime values of polynomials, currently far outside of our reach.}  in the limited range $y>x^{d-1+\epsilon}$.  Unconditionally, this conjecture has only been proved for linear polynomials.

For our purposes we will state a weaker conjecture
\begin{conj}
 For any non-constant polynomial $P \in \Z[x]$  we have that
 $$\limsup_{x,y \geq 2} \frac{\Psi_P(x,y) \log y} y=\infty.$$
\end{conj}
\begin{thm}
 Conjecture 3 implies Conjecture 4.
\end{thm}
\begin{proof}  Choose $x=y$ in Conjecture 3.
We see that
$$
\frac{\Psi_P(x,x)} x \gg \rho(d) \gg 1,
$$
for fixed $d$ by the fact that the Dickman-function $\rho(d)>0$ is strictly positive. Thus we get that
$$
 \frac {\Psi_P(x,x) \log x} x \gg \log x, \qquad  \text{and} \qquad
\limsup_{x,y \geq 2} \frac{\Psi_P(x,y) \log y} y=\infty.$$
\end{proof}

\section{Smooth numbers and the Dubickas conjecture}

While Conjecture 2 deals with smoothness of integer values of polynomials, it is intimately connected with smoothness of the algebraic numbers $n +\alpha$. 
 We will let $\displaystyle \alpha=\frac{\gamma}{\beta}$ where $\beta,\gamma$ are co-prime algebraic integers. This allows us to deal with algebraic integers $n\beta+\gamma$ rather than algebraic numbers $n+\alpha$, since
\begin{gather} \label{bgdef}
  n+\alpha =\frac{n \beta+\gamma} {\beta}.  
\end{gather}
The norm of $n\beta+\gamma$ will be
 %$P(x)$ is an irreducible polynomial in $\Z[x]$, then the Norm of $\alpha+n$ will be
\begin{gather} \label{Norm} \Norm(n \beta+\gamma)= P(n), \end{gather}
 for some  irreducible polynomial $P$ in $\Z[x]$, and if $P(n)$ is smooth and have prime factors less than or equal to $y$ then for any prime factorization of the algebraic integer
\begin{gather} \label{fact}
  n \beta+\gamma =u_1^{k_1} \cdots u_r^{k_r} q_{k_1}^{n_1} \ldots q_{k_m}^{n_m}.
\end{gather}
where the norm of a prime element $q_{k_i}$ is a power of some prime less than or equal to $y$.
Here $q_j$ are non-associated prime elements in the algebraic number field,  $u_1,\ldots,u_r$ are units that generates the group of units and $k_i,n_j$ are 
natural numbers. The study of the group of units is an interesting topic. 
However, in our case it will be sufficient to use that it is finitely generated. Except the more complicated structure of the group of units,  the main difference between the classical integer case
 and the number field case is that prime factorization is not neccessarily unique. In our application, this will not be a problem, rather the opposite, since non uniqueness of prime factorization 
will only lead to more relations, rather than fewer.

\begin{thm}
 Conjecture 4 and thus also Conjecture 3 implies an affirmative answer to Question and that Conjecture 2 is false. 
\end{thm}

\begin{proof}
%If we know say that a positive proportion of the numbers $\{n+\alpha \}$ only has prime factors with norm less than $y$ then it follows 
Let the algebraic integers $\beta,\gamma$ and the polynomial $P$ be defined by Eq. \eqref{bgdef} and Eq. \eqref{Norm} respectively. From the Landau prime ideal  theorem for number fields \cite{Landau} it follows that the number of  prime elements (up to multiplication with units)  with norm equal to a prime power of a prime less 
than or equal to $y$ in a number field,  is  strictly less than $C\frac y{\log y}$ for $y \geq 2$ and some $C>0$ depending on the field.  By Conjecture 4 we can choose some sufficiently large $x$ and $y$  such that
\begin{gather} \label{Orlow}
   \Psi_P(x,y) > C \frac y {\log y} +r+3,
\end{gather}
where $r$ denotes the cardinality of a minimal set of units that generates the group of units.
 Consider the vector space $V$ over $\Q$ spanned by the set
$$M= \{2 \pi i\} \cup \{\log(n+\alpha):0 \leq n < x \land P^+(P(n)) \leq y\}.$$
Here $P^+(m)$ denote the largest prime factor of $m$. It is  clear that the cardinality \begin{gather} \label{Orlow2} 
\# M= \Psi_P(x,y)+\delta \geq \Psi_P(x,y),
\end{gather}
where $\delta \in \{0,1\}$. Since by Eq. \eqref{bgdef} and Eq \eqref{fact} each element in $V$ can also be written as a linear combinations of the logarithm of prime elements and logarithms of the units, $\log \beta$ and $2 \pi i$, 
a set of vectors $B$ that spans a space $W$ such that $V \subset W$ can be chosen as 
$$
 B = \{2 \pi i,\log \beta\} \cup\{ \log u_1,\ldots,\log u_r,\log q_1,\ldots,\log q_{m} \}. 
$$
Here $q_1,\ldots,q_m$ are  the non-associated prime elements with norm a power of a prime that is less than or equal to $y$. It is clear that $m <  C \frac y {\log y}$. The cardinality of $B$ is hence certainly
 less than $C \frac y {\log y}+r+2$. Thus we have that $\dim_\Q V \leq \dim_\Q W < C \frac y {\log y}+r+2$. Since $M \subset V$ by Eqs. \eqref{Orlow},\eqref{Orlow2} contains more elements than the dimension of $V$ there must exists some nontrivial linear relation between the elements in $M$ and thus they are linearly dependent over $\Q$ and the set $\{(n+\alpha): 0 \leq n < x:P^+(P(n))<y\}\subset \N+\alpha$ and also $\N+\alpha$ is  multiplicative dependent. 
\end{proof}

\begin{rem}
  The limited range of $y$ that Martin \cite{Martin} proved Conjecture 3 under Shinzel's hypothesis H, does not help us to prove Theorem 3. Thus even under this very strong hypothesis, the answer to Question does not follow. 
\end{rem}

The following corollary is a result of Dubickas \cite{Dub}. While he used elegant explicit constructions involving Pell's equation, we will give a different more indirect proof and show how his result follows from the theory we developed.
\begin{cor}
  The answer to Question if $\alpha$ is a quadratic irrational is {\em Yes}.
\end{cor}
\begin{proof}
 Dartyge-Martin-Tenenbaum \cite[Theorem 1.3]{DarMarTen} proved a result which for one quadratic polynomial $P$ and $x=y$ specializes to
 $$
   \Psi_P(x,x) \gg_\beta \frac{x}{(\log x)^{\beta}}, \qquad (\beta>\log 4-1).
$$
Let us choose $\beta=1/2$ which is admissible since $1/2>\log 4-1=0.386$.
We see that 
$$\frac{\Psi_P(x,x) \log x} x \gg \sqrt{\log x}. $$
Thus
$$\limsup_{x,y \geq 2} \frac{\Psi_P(x,y) \log y} y =\infty,$$
and Conjecture 4 is true for the quadratic number field case. Now, by Theorem 3 we know that Conjecture 4 for quadratic polynomials implies an affirmative answer to Question in the quadratic number field case.
\end{proof}
In a similar but simpler way, we can also prove the case of $\alpha$ a rational number. In fact Conjecture 3 is known to hold for polynomial of degree 1 (see \cite{Granville}).

\section{Quantitative measures of dependence for shifted algebraic numbers}

The same reasoning as in the proof of Theorem 3 will in fact lead to a plausible conjecture for the dimension of $\Span_\Q \{\log(n+\alpha):0\leq n \leq x\}$.
\begin{conj}
 Let $\alpha$ be an algebraic number of degree $d \geq 2$. Then
  $$\dim  \Span_\Q \{\log(n+\alpha),0 \leq n < x\}\sim (1-\rho(d))x.$$
\end{conj}
Worley \cite{Worley} mentioned that a limited amount of computing indicates that $1-\rho(2)$ is the correct limit for the case when $\alpha^2+\alpha+1=0$, and that $1-\rho(3)$ is a plausible limit for the case when $\alpha^3+\alpha^2+\alpha+2=0$.

Conjecture 5 improves on Cassels and Worley's results. For example, in the quadratic number field case Cassels gives a lower bound of $0.51x$ in Conjecture 5, while Conjecture 5 implies the estimate $x\log 2=0.71x$. Similarly for large degree $d$ known asymptotics for the Dickman-function \cite[Eq 1.6]{Granville} gives us the lower bound $1-d^{-d+o(d)}$ which is certainly better than Worley's lower bound $1-\frac 1 {2d}+O(d^{-3/2})$.

 While we do not have an unconditional proof of  Conjecture 5 
we will however manage to prove this under Martin's conjecture, Conjecture 3, and some assumption on the number field.

\begin{thm}
  Conjecture 3 implies Conjecture 5, whenever the ring of integral elements in the number field generated by $\alpha$ is a Unique Factorization Domain.
\end{thm}
The number fields studied by Worley have class number one, so their ring of integers is a Unique Factorization Domain. Thus this result gives some evidence in favor of the observation of Worley. Theorem 4 is an immediate consequence of the following  Theorem.
\begin{thm}
  Suppose that for the polynomial $P$ defined by Eq. \eqref{Norm} we have that
  \begin{gather*}
     \liminf_{x \to \infty} \frac{\Psi_P(x,x)} x=B. \\
\intertext{Then}
       \dim \Span_\Q \{\log(n+\alpha):0 \leq n <x  \} \leq (1-B)x+o(x). \\ \intertext{If furthermore the ring of integers of the number field generated by $\alpha$ is a Unique Factorization Domain and}
 \limsup_{x \to \infty} \frac{\Psi_P(x,x)} x =U, \\ 
 \intertext{Then we have the corresponding lower bound}
       (1-U)x+o(x) \leq \dim \Span_\Q \{\log(n+\alpha):0 \leq n < x \}. 
\end{gather*}
\end{thm}

\begin{proof}
Let the algebraic integers $\beta,\gamma$ and the polynomial $P$ be defined by Eq. \eqref{bgdef} and Eq. \eqref{Norm} respectively.  In this proof we call an algebraic integer smooth if all its 
prime elements that occurs in some factorization \eqref{fact} have norm a power of a prime less than or equal to $x$.

{\em Upper bound:} 
 We use the same proof method as we used to prove Theorem 3. By the conditions of the theorem, for any $\epsilon>0$ and sufficiently large $x$ we have that
$$\Psi_P(x,x)\geq (B-\epsilon)x.$$
Now there are at most $C\frac x {\log x}$ number of non-associated prime elements in the number field with norm a power of some prime less than $x$, for some $C>0$. By similar reasoning as in the proof of Theorem 3, where we also need to consider  $2\pi i$  and $\log \beta$ respectively, this
 means that dimension of the space that is spanned by the smooth numbers $\log(n\beta+\gamma)$ with $0 \leq n <x$ is no more that $C\frac x {\log x}+r+1$, and 
the dimension of the space that is spanned by the corresponding numbers $\log(n+\alpha)$ with $0 \leq n <x$  is no more that $C\frac x {\log x}+r+2$. The rest of the numbers $n \beta+\gamma$ that are not smooth  are at most $(1-B+\epsilon)x$. This gives us the following inequality
$$
 \dim \Span_\Q \{\log(n+\alpha):1 \leq n < x\} \leq (1-B+\epsilon)x+\frac {Cx} {\log x}+r+2.
$$
Since $\epsilon$ is arbitrarily small this completes the proof of the upper bound. %\prod_{j=1}^n x_j^{k_j}=1

{\em Lower bound:} 
By the conditions of the Theorem we have for any $\epsilon>0$ and sufficiently large $x$ (here it is convenient to assume that $ \left |\Norm(\beta) \right |<x$) that
 $$
           \Psi_P(x,x) \leq (U+\epsilon)x.
  $$
 This means that there are at least $(1-U-\epsilon)x$ non-smooth numbers among the integers $P(n)$ where $0 \leq n < x$ and thus also among the algebraic integers $n\beta+\gamma$ where $0 \leq n <x$. Let us now assume that $P(n)$ has a prime $p>x$.
 That $p|P(n)$ implies that $q|(n\beta+\gamma)$ for some prime element $q$ in the number field of norm $p^k$. However such a prime element $q$ must divide just one number $(n\beta+\gamma)$ for $0 \leq n < x$, 
because if $q|(n_1 \beta+\gamma)$ and $q|(n_2 \beta+\gamma)$ where $0 \leq n_1<n_2<x$ then $q|((n_2 \beta+\gamma)-(n_1 \beta +\gamma))$ and $q| (n \beta)$, where $n=n_2-n_1$. This is not possible since
 $1 \leq n<x \leq p$ and $0< \left|\Norm(\beta) \right|<x \leq p$.  By unique factorization it follows that the $\log q_i$ where $q_i$ are non-associated prime elements are linearly independent over $\Q$ and from this it follows that the numbers $\log(n+\alpha)$ where $P(n)$ is non-smooth are linearly independent over $\Q$. The dimension result follows from the existance of at least $(1-U-\epsilon)x$ non-smooth numbers.
\end{proof}

\section{Remaining problems and remarks}
There is one obvious problem
\begin{prob}
 Prove Theorem 4 and 5 without assuming unique factorization.
\end{prob}
This seems as an approachable problem. There are many other remaining problems in the field. For example, there are some results of Hmyrova \cite{Hmy1,Hmy2} that gives upper bound for $\Psi_P(x,y).$
Unfortunately she just proves non-trivial results when $y$ is small compared to $x$, and the region where it is neccessary to obtain results in order to apply Theorem 5 is when $x$ is the same order as $y$. However Timofeev \cite{Timofeev}\footnote{This author was not able to find Timofeev's paper so this information is from mathscinet. It is not quite clear if his work is relevent} has also done some work in this direction, and eventually it might be possible to prove upper bounds for $\Psi(x,y)$ which by means of Theorem 5 might give sharper results than Worley's result for large degree $d$.

Also while we are sure that Conjecture 1 and Conjecture 3 should be quite difficult, it might be that Conjecture 4 is more approachable.

\begin{ack}
 Thanks go to Roma Ka{\v{c}}inskait{\.e}    and  Art{\=u}ras Dubickas who brought the problem to my attention at Journ\'ees Arithm\'etiques, Vilnius, June 27th-July 1'st 2011.
\end{ack}

\bibliographystyle{plain}
\def\polhk#1{\setbox0=\hbox{#1}{\ooalign{\hidewidth
  \lower1.5ex\hbox{`}\hidewidth\crcr\unhbox0}}}

\end{document}